\documentclass{amsart}
\usepackage{amscd}
\usepackage[arrow,matrix]{xy}
\usepackage{graphicx}
\usepackage{amsmath}
\usepackage{amsmath, latexsym, amssymb}
\usepackage{color}

\begin{document}
\hsize 125 mm
\vsize 185 mm
\input xypic
\numberwithin{equation}{section}
\theoremstyle{plain}
\newtheorem{lemma}{Lemma}[section]
\newtheorem{proposition}[lemma]{Proposition}
\newtheorem{theorem}[lemma]{Theorem}
\newtheorem{corollary}[lemma]{Corollary}
\newtheorem{problem}[lemma]{Problem}

\theoremstyle{definition}
\newtheorem{definition}[lemma]{Definition}
\newtheorem{remark}[lemma]{Remark}
\newtheorem{example}[lemma]{Example}

\DeclareGraphicsRule{.tif}{png}{.png}{`convert #1 `dirname #1`/`basename #1 .tif`.png} 
\newcommand{\R}{{\mathbb R}}
\newcommand{\C}{{\mathbb C}}
\newcommand{\F}{{\mathbb F}}
\renewcommand{\O}{{\mathbb O}}
\newcommand{\Z}{{\mathbb Z}} 
\newcommand{\N}{{\mathbb N}}
\newcommand{\Q}{{\mathbb Q}}
\renewcommand{\H}{{\mathbb H}}

\newcommand{\Aa}{{\mathcal A}}
\newcommand{\Bb}{{\mathcal B}}
\newcommand{\Cc}{{\mathcal C}}    
\newcommand{\Dd}{{\mathcal D}}
\newcommand{\Ee}{{\mathcal E}}
\newcommand{\Ff}{{\mathcal F}}
\newcommand{\Gg}{{\mathcal G}}    
\newcommand{\Hh}{{\mathcal H}}
\newcommand{\Kk}{{\mathcal K}}
\newcommand{\Ii}{{\mathcal I}}
\newcommand{\Jj}{{\mathcal J}}
\newcommand{\Ll}{{\mathcal L}}    
\newcommand{\Mm}{{\mathcal M}}    
\newcommand{\Nn}{{\mathcal N}}
\newcommand{\Oo}{{\mathcal O}}
\newcommand{\Pp}{{\mathcal P}}
\newcommand{\Qq}{{\mathcal Q}}
\newcommand{\Rr}{{\mathcal R}}
\newcommand{\Ss}{{\mathcal S}}
\newcommand{\Tt}{{\mathcal T}}
\newcommand{\Uu}{{\mathcal U}}
\newcommand{\Vv}{{\mathcal V}}
\newcommand{\Ww}{{\mathcal W}}
\newcommand{\Xx}{{\mathcal X}}
\newcommand{\Yy}{{\mathcal Y}}
\newcommand{\Zz}{{\mathcal Z}}

\newcommand{\zt}{{\tilde z}}
\newcommand{\xt}{{\tilde x}}
\newcommand{\Ht}{\widetilde{H}}
\newcommand{\ut}{{\tilde u}}
\newcommand{\Mt}{{\widetilde M}}
\newcommand{\Llt}{{\widetilde{\mathcal L}}}
\newcommand{\yt}{{\tilde y}}
\newcommand{\vt}{{\tilde v}}
\newcommand{\Ppt}{{\widetilde{\mathcal P}}}
\newcommand{\bp }{{\bar \partial}} 
\newcommand{\ad}{{\rm ad}}
\newcommand{\Om}{{\Omega}}
\newcommand{\om}{{\omega}}
\newcommand{\eps}{{\varepsilon}}
\newcommand{\Di}{{\rm Diff}}

\renewcommand{\a}{{\mathfrak a}}
\renewcommand{\b}{{\mathfrak b}}
\newcommand{\e}{{\mathfrak e}}
\renewcommand{\k}{{\mathfrak k}}
\newcommand{\pg}{{\mathfrak p}}
\newcommand{\g}{{\mathfrak g}}
\newcommand{\gl}{{\mathfrak gl}}
\newcommand{\h}{{\mathfrak h}}
\renewcommand{\l}{{\mathfrak l}}
\newcommand{\sm}{{\mathfrak m}}
\newcommand{\n}{{\mathfrak n}}
\newcommand{\s}{{\mathfrak s}}
\renewcommand{\o}{{\mathfrak o}}
\newcommand{\so}{{\mathfrak so}}
\renewcommand{\u}{{\mathfrak u}}
\newcommand{\su}{{\mathfrak su}}

\newcommand{\ssl}{{\mathfrak sl}}
\newcommand{\ssp}{{\mathfrak sp}}
\renewcommand{\t}{{\mathfrak t }}
\newcommand{\Cinf}{C^{\infty}}
\newcommand{\la}{\langle}
\newcommand{\ra}{\rangle}
\newcommand{\half}{\scriptstyle\frac{1}{2}}
\newcommand{\p}{{\partial}}
\newcommand{\notsub}{\not\subset}
\newcommand{\iI}{{I}}               
\newcommand{\bI}{{\partial I}}      
\newcommand{\LRA}{\Longrightarrow}
\newcommand{\LLA}{\Longleftarrow}
\newcommand{\lra}{\longrightarrow}
\newcommand{\LLR}{\Longleftrightarrow}
\newcommand{\lla}{\longleftarrow}
\newcommand{\INTO}{\hookrightarrow}
\newcommand{\supp}{\operatorname{supp}}

\newcommand{\QED}{\hfill$\Box$\medskip}
\newcommand{\UuU}{\Upsilon _{\delta}(H_0) \times \Uu _{\delta} (J_0)}
\newcommand{\bm}{\boldmath}

\title[Poisson smooth structures on stratified symplectic spaces]{\large   Poisson smooth structures on stratified symplectic spaces}
\author{H\^ong V\^an L\^e, Petr Somberg and Ji\v ri Van\v zura} 


\medskip

\abstract   In this   paper we  introduce the notion of a smooth structure  on a  stratified  space, the notion of a Poisson smooth structure  and the notion of a weakly symplectic smooth structure on a   stratified symplectic space, refining the concept of a  stratified  symplectic Poisson  algebra   introduced by
Sjamaar and Lerman. We show that these smooth spaces possess   several important properties, e.g. the existence of  smooth partitions of unity.
Furthermore,  under  mild  conditions  many   properties of  a symplectic manifold can be extended to a   symplectic stratified space provided with a smooth Poisson structure,
 e.g. the existence  and
uniqueness of a Hamiltonian flow, the isomorphism between  the Brylinski-Poisson homology and the de Rham homology,  the
existence of a Leftschetz decomposition on
a symplectic  stratified space.    We give many examples   of  stratified symplectic spaces  possessing a Poisson smooth structure  which is also weakly symplectic.
\endabstract

\maketitle
\tableofcontents

{\it  AMSC: 51H25, 53D05, 53D17}


\section{Introduction}

Many classical problems on various classes of topological spaces reduce to the quest 
for its appropriate functional structure. Examples of topological 
spaces we are interested 
in comprise stratified spaces equipped with an additional structure of geometrical origin.  
Due to the lack of canonical notion of the sheaf of smooth (or analytic) functions on 
such spaces one has  to define  a  smooth structure with all derived smooth (or analytical)
notions  such that the obtained smooth structure   satisfies   good formal   properties. 

In this paper we  study  smooth structures on    stratified spaces, developing the ideas  in \cite{LSV2010}. We observe  that many properties  of smooth structures
on pseudomanifolds  with  isolated  conical singularities  also hold  for a larger  class of  locally trivial spaces with singularities  with cone as typical fiber, if one poses a mild, natural, local condition on these smooth structures (Definition \ref{smooths}). In this extension  we have to take care of (possibly disconnected)  regular strata of different dimensions. 
A large part of our note concerns with  compatible smooth structures on stratified  symplectic spaces.  Stratified symplectic spaces appear abundantly in geometry and mathematical physics  \cite{Hueb2004}, \cite{Hueb2007}, \cite{Karshon1992}.
They  are  subjects of intensive  study in symplectic geometry
since  nineties   \cite{SL1991}, \cite{BL1997}, \cite{CR2002},  \cite{Pflaum2000}, \cite{Hueb2004}.  In  \cite{Beauville2000},  \cite{FU2006}  and many other papers in this direction the authors considered complex manifolds $M^{2n}_\C$ with a holomorphic symplectic form $\om^2$, which
can be turned into  a  symplectic form $\tilde \om^2$ on  differentiable manifolds $M^{2n}_\C$  of real dimension $4n$ by setting $\tilde \om^2 :=Re\, (\om ^2) + Im\, (\om^2)$.  The theory of smooth structures on  general stratified
symplectic spaces  has  been  introduced by   Sjamaar and    Lerman  \cite{SL1991},  then developed by  many others \cite{BL1997},  \cite{Pflaum2000}, \cite{Hueb2004} (in  Remarks \ref{rsmooth}.3, \ref{compsl} and Example \ref{quot} we compare  our notion of  a smooth structure  on a stratified symplectic space with  that given  by Sjamaar-Lerman).  Pflaum  introduced smooth structures on stratified  spaces by means of a maximal atlas (Remark \ref{rsmooth}.3), and with reference to a smooth structure, he developed extensions to stratified spaces of standard differential geometric concepts  \cite{Pflaum2000}.  Huebschmann  refined  the notion of symplectic  stratified spaces  by  introducing the notion of stratified polarization \cite{Hueb2004}. 
 Our axiomatic  approach  to the notion of  a smooth structure uses  heavily  the notion of $C^\infty$-algebras defined in \cite{MR1991}. We refer the reader to \cite{GS2003}  for  the notion of  $C^\infty$-differentiable space developed by Gonsalez and de Salas, which  also  grew   from the notion of  $C^\infty$-algebra, but  they  did not  treat  stratified spaces.  
In   our study    we discover the  existence of    a   class  of symplectic stratified spaces which can be supplied with a  Poisson  and weakly symplecitc smooth structure. 
This  class  is large enough to encompass  basic examples of symplectic singular spaces considered in \cite{SL1991} (Example \ref{quot}, Proposition \ref{quote}) as well as many important singularities on  the closure of  adjoint orbits  of nilpotent elements of  complex Lie algebras (Example \ref{nil}, Proposition \ref{pois1}).

The structure of this  paper is as follows.
In chapter 2 we  introduce the notion of a  smooth structure on  
a  stratified space (Definition \ref{smooths})  and compare our concept  of a smooth structure with some   other concepts (Remarks \ref{rsmooth}.2-3).
We  prove that  a smooth stratified space possesses   several important properties, e.g. the existence of smooth  partitions of unity  (Lemma \ref{main1}, Proposition \ref{fin} and Corollary \ref{germ}), which will be needed in later sections (Remarks \ref{hsl}.2,  \ref{cav}).  In chapter 3 we  study natural smooth structures on stratified   symplectic spaces (Definition \ref{spos}).
We show that, under mild conditions,    stratified symplectic spaces $(X, \om)$ equipped with a  Poisson smooth structure   possess  
a variety of basic properties of  smooth symplectic manifolds, e.g. 
the existence and uniqueness of a Hamiltonian flow, the isomorphism between the Brylinski-Poisson 
homology and  the de Rham homology,  and the existence of  a Leftschetz decomposition (Theorems \ref{hodge},  \ref{ham}, \ref{harm},  Proposition  \ref{hlv}).  We also  show many examples of  weakly symplectic smooth structures  and Poisson smooth structures (Examples \ref{quot},  \ref{nil}).

\section{Stratified spaces  and   their  smooth   structures}\label{s1}
In this section we recall the notion of a stratified space following  Goresky's and MacPherson's  concept \cite[p.36]{GM1988},  \cite[\S 1]{SL1991}
(Definition \ref{strati}). Then we introduce the notion of a  smooth structure on a   stratified space (Definition \ref{smooths}).  Our  concept of  a smooth structure
on a stratified space is a natural extension  of our concept  of a smooth structure on a pseudomanifold with isolated conical singularities given in \cite[\S 2]{LSV2010}, using the key notion of a product smooth structure (Remark \ref{rsmooth}.2). (The notion of a product smooth structure has   been introduced by Mostow in \cite[\S 3]{Mostow1979} and has been used by many authors).  We prove  several important properties of a smooth structure on a stratified space, e.g. the existence of smooth partitions of unity and its consequences  (Lemma \ref{main1}, Proposition \ref{fin}, Corollary \ref{germ}), the existence of a locally smoothly contractible, resolvable smooth structure on  pseudomanifolds with edges (Lemma \ref{canres}).  We show that a resolvable smooth structure satisfying a mild condition  is not finitely generated  (Proposition \ref{infg}).  We introduce the notion of  a smooth differential form  and  a smooth Zariski  vector field (Definition \ref{smoothf}).  We compare our concept of a smooth structure  with some other
concept (Remark \ref{rsmooth}.3).

\subsection{Stratified  spaces} We begin with the notion of a decomposed space.

\begin{definition}\label{decom}(\cite[p.36]{GM1988}, \cite[Definition 1.1]{SL1991})  Let $X$  be a Hausdorff and  paracompact topological space 
 and let $\Ss$ be a partially ordered set  with ordering denoted
by $\le$.  {\it An $\Ss$-decomposition of $X$} is a  locally finite collection of disjoint locally closed manifolds
$S_i\subset X$ (one for each $ i \in \Ss$) called {\it strata} such that 

1) $ X = \cup _{i \in \Ss} S_i$;

2) $S_i \cap \bar  S_j  \not = \emptyset \LLR  S_i \subset \bar S_j \LLR i \le j .$

We define {\it the depth of a stratum $S$} as follows 
$$depth _X\, S : = \sup \{n | \: \text{ there exist  strata } S = S_0 < S_1 < \cdots  < S_n\}.$$

We define {\it the depth of $X$} to be the number $depth \, X: = \sup_{i \in \Ss} depth_X\,S_i.$ {\it The dimension of $X$} is defined
to be the maximal dimension of its strata.
\end{definition}
As in \cite{SL1991} we  consider only   finite-dimensional  decomposed spaces  $X$. We  call   a stratum $S\subset X$ {\it singular},  if  there is another stratum $S'\subset X$ such that $S\subset \bar S'$. Otherwise  $S$ is called {\it regular}.  We call
$$X^{reg}: = \{ \cup  S | S \text{ is a  regular  stratum of }   X\}$$
{\it the regular part (component) of $X$}.  Clearly $\overline{X^{reg}} = X$.

Now we  specify  a subclass  of decomposed spaces  consisting of   locally trivial   spaces with cone as typical fiber.  Recall that   a cone $cL$  over  a  topological space  $L$ is the topological space $L \times [0, \infty ) /L\times \{0\}$.  Let $[z,t]$ denote the image of $(z,t)$ in a cone  $cL$  under the projection $\pi: L\times [0,\infty) \to cL$. We denote by $cL (\eps)$  the open subset $\{ [z, t] \in cL\| \:  t< \eps \}$.

\begin{definition}\label{strati}   (\cite{GM1980}, \cite[Definition 1.7]{SL1991}) A decomposed space $X$ is called  {\it a stratified space} if the strata of $X$ satisfy the following condition defined recursively.
Given a point  $x$ in a stratum $S$ there exist  an open neighborhood  $U(x)$ of $x$ in $X$, an open ball $B_x$ around $x$ in $S$, a  compact stratified  space $L$, called the link of $x$, and a homeomorphism $\phi_x: U(x) \to B_x \times cL(1) $ that preserves the decomposition.
\end{definition}

{\begin{remark} \label{rem:neu} 1.  It is well-known  that   a Whitney stratified subspace in $\R^n$  is a stratified  space  in sense of Definition \ref{strati} 
\cite{Mather1970}. %

2. In the literature  there are     different   concepts of  a stratified space,  see e.g.  \cite[Chapter 1]{Pflaum2000} for a discussion. 

\end{remark}

\begin{example}\label{edge} 1.
Among important  examples  of stratified spaces  of depth 1 are pseudomanifolds with edges, see e.g. \cite{NSSS2004}. Let us recall the definition of a pseudomanifold with edges. Suppose that $M$ is a compact connected smooth manifold with boundary $\p M$,
and suppose that $\p M$ is the total space of a smooth locally trivial bundle $\pi: \p M \to N$
over a closed smooth base $N$ whose fiber is a closed smooth manifold.  The topological space  $X$ obtained  by gluing  $M$ to $N$    with help of $\pi$ (i.e. the points in each fiber $\pi^{-1}(s)$ are identified with $s \in N$)  is called {\it a pseudomanifold with edges} corresponding to the pair $(M, \pi)$. The natural surjective map $M \to X$ which is the identity on $M \setminus \p M$ is denoted by $\bar \pi$.  In general, $N$ need not be connected, and the  connected components of $N$ are called {\it edges} of $X$.  Clearly, $X= (X \setminus N)\cup N$ is a decomposed space, moreover $X\setminus N$ is an open connected stratum of $X$. Now we will prove that $N$ satisfies the condition in Definition \ref{strati}. For $s \in N$ let $L$ be the fiber $\pi ^{-1} (s) \subset \p M$  and $B$ be an open  neighborhood of $s$ in $N$  such that $\pi ^{-1} (B) = B \times \pi ^{-1} (s)$.
Let $\pi ^{-1} (N) _\eps$ be a collar neighborhood  of the boundary component $\pi ^{-1} (N)\subset \p M$  in $M$  provided with  a  trivialization $(p, t) : \pi^{-1} (N)_\eps  \to \pi^{-1} (N) \times [0, \eps)$, where $p$ is 
a  smooth retraction $\pi ^{-1} (N)_\eps\to  \pi^{-1}(N)$.   Set $U(s) :=\bar \pi (p^{-1} \circ \pi ^{-1} (B))$.  We define    a trivialization $\phi_s : U(s) \to B \times cL(1)$  by $\phi _s (x) = (\pi \circ p (x), [t, p(x)])$. This shows that $X$ is a stratified space of depth 1.
An important  class of pseudomanifolds with  edges consists of {\it   pseudomanifolds $X$ with isolated conical singularities}, if its edges $X_i$   of $X$ are just points $s_i$  of $X$.

2.  If $L$ is a compact stratified space, then  the cone $cL$ is also a stratified space.

3. If $L_1$ and $L_2$ are stratified spaces, then $L_1 \times L_2$ is a stratified space.  In particular,
a product of cones $ cL_1  \times cL_2  = c ( L_1 \times L_2 \times [0,1])$  has the following decomposition:
$c (L_1 \times L_2 \times [0,1]) = \{ pt \} \cup L_1 \times (0,1) \cup L_2 \times (0,1)  \cup L_1 \times L_2 \times (0,1)$. 

\end{example}

\subsection{Smooth structure on stratified spaces  and its simplest properties}
We now introduce the notion of a smooth structure on a   stratified space  $X$, which is a  natural extension
of our notion of a smooth structure on a pseudomanifold with isolated conical singularities (pseudomanifold w.i.c.s.) in \cite{LSV2010} (Remark \ref{rsmooth}.2).

As the notion of  a stratified space  $X$ is  defined  inductively  on the depth of  $X$,  the  notion
of  a smooth structure on $X$  is also defined inductively  on the depth on $X$.

 Since $X$ is  locally modelled as
a product $B\times cL(1)$, we  define the  notion of a smooth structure on   the product $B \times cL(1)$
inductively  on the depth of $L$. 
  More generally, we  introduce the notion of a smooth structure on a stratified space $X$ of depth $k$, recursively
on $k$. The key notion  is  a  product smooth structure (Definition \ref{smooths}.2).

For a  stratum $S \subset  X$  we  consider  the following associated spaces.
\begin{itemize} 
\item $C^\infty _u (S)$ - the space of  all  usual  smooth  functions  on $S$.
\item $C^{\infty}_{u, 0}(S) : = C^ \infty _u(S) \cap  C^0_0 (S)$, where   $C^0_0(S)$ is the space
of all  continuous  functions with  compact support.
\item For  $f \in C^0_0(S)$,  let $j_*(f)\in C^0 _0 (X)$ denote   the unique  extension of $f$ such that $j_*(f) = 0$ if $ x \not \in S$.
\end{itemize}
\begin{definition}\label{smooths} (cf. \cite[Definition 2.2]{LSV2010}) {\it  A smooth structure} on  a stratified space $X$  of depth $k$ is a choice of   a $\R$-subalgebra $C^\infty (X)$ of the algebra $C^0(X)$ of continuous real-valued functions on $X$ that     satisfy the  following   properties.

1.  $C^\infty (X)$ is a   germ-defined $C^\infty$-ring.

2. For any $x \in X$ there exists a local  trivialization $\phi_x: U(x)\to B_x \times cL(1)$  which is a local diffeomorphism of stratified  spaces, i.e.,  $C^\infty(U) =
\phi_x^* (C^\infty (B \times cL(1)))$, where  $C^\infty (B \times cL(1))$  is  {\it a   product smooth structure}. In other words,  
$C^\infty (B \times cL(1))$   is the  germ-defined $C^\infty$-ring   whose  sheaf $SC^\infty (B \times cL(1))$ is generated by $\pi_1^* ( SC^\infty (B))$ and $\pi_2 ^* (SC^\infty (cL(1)))$, where $\pi_1$ and $\pi_2$ are the  projections  from $B\times cL(1)$ to $B$ and $cL(1)$ respectively (cf. \cite[\S 3]{Mostow1979}).

3. A  smooth structure  $C^\infty (cL(1))$ on the cone  over  a  compact stratified  space $L$  must satisfy the following  two additional properties:
(3a) $C^\infty (cL(1))|_{ L \times (0,1)} \subset C^\infty (L \times (0,1))$, and   (3b) 
$j_*[\pi_1^*(C^\infty_{u,0}(0,1))]
\subset C^\infty (c(L(1))$, where $\pi_1: L \times (0,1) \to (0,1)$ is the projection. 
\end{definition}

\begin{lemma}\label{main1}  Any  smooth structure $C^\infty(X)$ on a stratified space $X$ satisfies  the following properties.
 
1. $ C^\infty (X)_{|S} \subset C^\infty_u (S)$  for each stratum $S$  of $X$.

2.  (cf. \cite[Lemma 2.1]{LSV2010}) $C^\infty(X)$ is  partially invertible in the following  sense.  If $f  \in  C^\infty (X)$ is nowhere vanishing, then
$ 1/f  \in C^\infty  (X)$.
\end{lemma}

\begin{proof}  We prove the first  assertion  of Lemma \ref{main1} using   induction  on the dimension of $X$.  Clearly  this assertion is valid if $\dim X = 1$.  Since $C^\infty (X)$  is   germ-defined, it suffices  to prove   Lemma \ref{main1}.1
locally. Hence,  w.l.o.g., by Definition  \ref{smooths}.2 we can assume that $X= B \times cL(1)$ and $C^\infty (X)$ is a product smooth structure.
 
First, we consider  the case $S = B\times [L, 0]$.  By Definition \ref{smooths}.2, any smooth function on $X$ is locally  written as $G(f_1, \cdots f_n, h_1, \cdots h_k)$ where $G \in C^\infty (\R^{n +k})$, $f _i \in C^\infty (B)$, $h _i \in C^\infty (cL(1))$. Since the restriction of $h_i$ to $B $  is constant,  $G(f_1, \cdots, f_n, h_1, \cdots, h_k)_{|S}$ is a smooth function on   $B$. Now assume that  $S= B \times S_L \times  (0,1)$, where $S_L$ is a   stratum of $L$. Using the condition  (3a) in Definition \ref{smooths},  we can assume that $X = B \times L \times (0,1)$ and  $C^\infty (X)$  is    generated by $\pi^*_1(SC^\infty (B))$, $\pi^*_2(SC^\infty(L))$, $\pi^*_3(SC^\infty_u(0,1))$,
where $\pi_1, \pi_2, \pi_3$  are the projections  from $X$  to $B$, $L$, $(0,1)$ respectively. Thus, any   function  in $C^\infty (X)$  is locally  of the form
$G(f_i, h_j, t)$ where $f _i \in C^\infty (B)$, $h _j \in C^\infty (L)$  and $t \in (0,1)$.  Using the  induction assumption, noting that $\dim L  \le  \dim X  -1$,
we    have $h_i |_{S_L} \in C^\infty _u (S_L)$.  
Hence $G|_{ B \times S_L  \times (0,1)}  \in C^\infty _u  ( B \times S_L \times (0,1))$. This completes the proof of  the first assertion of  Lemma \ref{main1}.

To prove the last assertion of Lemma \ref{main1} it suffices to show that  locally $1/f$ is   a smooth function.  Since $f\not =0$, shrinking a neighborhood  $U$ of $x$ if necessary,  we can assume that  there is an open interval $(-\eps, \eps)$  which has no intersection with $f(U)$.   Now there exists  a function $\psi :\R \to \R$  such that\\
a)$ \psi|_{f(U)} = Id$,\\
b) $(-\eps/2, \eps/2)$ does not intersect with $\psi (\R)$.

Clearly $G: \R \to \R$ defined by $G(x): = \psi (x) ^{-1}$ is a smooth function.  Note that $1/f(y) = G(f(y))$  for any $y \in U$. This completes the proof of the last  assertion of Lemma \ref{main1}.
\end{proof}

\begin{remark} \label{rsmooth} 1. Denote by $i$ the canonical inclusion $ X^{reg} \to X$. 
Since $X = \overline{X^{reg}}$, the kernel of $i^*:C^\infty (X) \to C^0 (X^{reg}) $ is zero. Lemma \ref{main1}.1 implies that
$i^*(C^\infty (X))$ is a subalgebra of $C^\infty_{u} (X^{reg})$. Roughly speaking, we can regard
$C^\infty(X)$ as a subalgebra of $C^\infty_{u} (X^{reg})$.

2. The condition (3b) in Definition \ref{smooths}    is a relaxing of the condition 3 
  of Definition  2.2 in \cite{LSV2010} for pseudomanifolds w.i.c.s.  that requires $j_*(C^\infty_0 M^{reg}) \subset  C^\infty (M^{reg})$.  In fact, in \cite{LSV2010} (and in the present note) we need  only   the (weaker)  condition (3b)  of Definition \ref{smooths}  for the   existence  of partition  of unity  and nothing more.

3. Our definition of a smooth structure  on a stratified space   is a refinement of   the definition   due to 
Sjammar and Lerman \cite{SL1991}, which  requires a smooth structure to satisfy only  Lemma \ref{main1}.1.   Pflaum introduces smooth structures by means of a maximal atlas; thus a smooth structure appears as an equivalence class of a system of local embeddings into suitable $\R^n$  \cite{Pflaum2000}. 
\end{remark}

We are  going to prove  the existence of  smooth partitions of unity, which is important for later applications (Remark
\ref{hsl}, Remark \ref{cav}).

\begin{proposition}\label{fin} (cf. \cite[Proposition 2.1]{LSV2010})
Let $\{U_i\}_{i\in I}$ be a locally finite open covering of $X$ such that each $U_i$ has a compact closure
$\bar{U}_i$. Then there exists a smooth partition of unity $\{f_i\}_{i\in I}$ subordinate to $\{U_i\}_{i\in I}$. 
\end{proposition}

\begin{proof}  This  is a local  statement, hence it suffices to prove for  the case $X = B \times cL(1)$.
Since   the smooth structure  on $B \times cL(1)$   is a product smooth structure, it  is  not hard to reduce
 Proposition \ref{fin} to  the case  $B $ is a point, i.e.  $X = cL(1)$ is a cone  over   a stratified space  $L$.  For  the  case $L$ is a smooth  manifold, by Remark \ref{rsmooth}.2, we have  proved  the corresponding assertion  in \cite[Proposition 2.1]{LSV2010}.  The  proof  of \cite[Proposition 2.1]{LSV2010}  can be repeated word-by-word for  the case  $L$ is a stratified space, using the  last two  conditions of Definitions \ref{smooths};  so we omit its  proof.
\end{proof}

The following  Corollary   is an immediate  consequence  of the  existence of    partition of unity, see. e.g. \cite[Lemma 2.11]{LSV2010}.

\begin{corollary} \label{germ}   Smooth functions on $X$ separate points on $X$.
\end{corollary}

Next, we introduce  the notion of the cotangent bundle  and the notion of the Zariski tangent bundle  of a  stratified  space $X$, in the same way as we did in  \cite{LSV2010}, which are similar to
the notions introduced in \cite{SL1991}, \cite[B.1]{Pflaum2000}.
Note that the germs of smooth functions
$C^\infty _x (X)$ is  a local $\R$-algebra  with  the unique maximal ideal $\sm _x$ consisting of functions that vanish 
at $x$. Set $T^*_x (X) : = \sm_x / \sm_x ^ 2$.  
Since  the following exact sequence 
\begin{equation}
0 \to \sm_x \to C^\infty _x \stackrel{j}{\to} \R \to 0
\end{equation}
splits, where $j$ is the evaluation map, $j(f_x) = f_x (x)$ for any $f_x \in C^\infty  _x$, the space $T_x^*X$ can be identified with the space of K\"ahler differentials of $C^\infty _x(X)$. The K\"ahler  derivation $d : C^\infty _x (X) \to T^*_x X$ is defined as follows:
\begin{equation}
d (f_x): = (f_x -  j^{-1} (f_x (x))  +\sm_x^2,
\end{equation}
where $j^{-1}: \R \to  C^\infty_x $ is the left inverse of $j$, see  e.g. \cite[Chapter 10]{Matsumura1980}, or \cite[Proposition B.1.2] {Pflaum2000}.    We  call $T^*_x X$  {\it  the cotangent space} of
$X$ at $x$. Its dual space
$ T^Z_x X: = Hom ( T^*_xX, \R)$ is  called {\it the Zariski tangent  space} of $X$ at $x$.
The union $T^*X : =\cup _{x \in X} T^* _x X$ is called {\it the cotangent bundle} of $X$.  The union $T^Z X: =\cup _{x \in X} T^Z _x X$ is called {\it the   Zariski tangent bundle} of $X$.

Let us denote  by $\Om^1_x (X)$ the $C^\infty _x (X)$-module   $C^\infty_x (X)\otimes_\R\sm_x/\sm_x^2$. We call  $\Om^1_x (X)$
{\it the germs of 1-forms at $x$}.  
Set $\Om^k _x (X):= C^\infty _x(X) \otimes _\R \Lambda ^k(\sm_x/\sm_x^2)$.  Then $\oplus _k\Om ^k_x  (X)$ is an exterior algebra with the following
wedge  product
\begin{equation}
(f\otimes _\R dg_1 \wedge \cdots \wedge dg_k)\wedge (f'\otimes_\R dg_{k +1}\wedge \cdots \wedge dg_l): = (f\cdot f')\otimes _\R dg_1 \wedge \cdots \wedge dg_l,
\end{equation}
where $f, f ' \in C^\infty _x$ and $d g_i \in T_x ^*M$.  

Note that the K\"ahler derivation $d : C^\infty_x (X) := \Om ^ 0 _x (X) \to \Om ^1_x (X)$ extends to the unique derivation $d: \Om ^k_x (X) \to \Om ^{k+1} _x (X)$  satisfying the Leibniz property. Namely we set
\begin{eqnarray}
d(f\otimes 1) := 1 \otimes df,\nonumber\\
d(f\otimes \alpha \wedge  g \otimes \beta): = d(f\otimes \alpha) \wedge g\otimes \beta + (-1) ^{deg\, \alpha} f\otimes \alpha \wedge  d(g \otimes \beta).\nonumber
\end{eqnarray}

\begin{definition} \label{smoothf}  1. (cf. \cite[\S 2]{Mostow1979})   A section $\alpha :X \to \Lambda ^k T^*(X)$  is  called {\it a smooth differential $k$-form}, if 	for each $x \in X$ there exists  a neighborhood
$U(x) \subset X$  of $x$ such that $\alpha (x)$ can be represented as $\sum _{i_0i_1\cdots i_k} f_{i_0}df_{i_1}\wedge
\cdots \wedge df_{i_k}$ for  some $f_{i_0}, \cdots ,f_{i_k} \in  C^\infty (X)$.

2. A section $V : X \to\Lambda^k T^Z X$  will be called {\it a smooth  Zariski  $k$-vector field}, if  for any  $\alpha \in \Om^k (X)$  the value $V(\alpha)$ is  a smooth function
on $X$.
\end{definition}

Denote  by $\Om(X)= \oplus _k \Om^k(X)$ the space of all smooth differential forms   on $X$. We identify  the germ at $x$
  of 
a $k$-form $\sum _{i_0i_1\cdots i_k} f_{i_0}df_{i_1}\wedge
\cdots \wedge df_{i_k}$   with the element $\sum _{i_0i_1\cdots i_k} f_{i_0}\otimes df_{i_1}\wedge
\cdots \wedge df_{i_k} \in \Om ^k_x(X)$.
Clearly the K\"ahler derivation $d$  extends to a map, also denoted by $d$, that sends  $\Om(X)$ to $\Om(X)$.

Now we  set
$$\Om _u(X^{reg}) := \Om (X^{reg}, C^\infty_u (X^{reg})). $$
Remark \ref{rsmooth}.1  implies immediately
 
\begin{lemma}\label{injf} The kernel $i^* :\Om (X)\to \Om_u (X^{reg})$ is zero.  Roughly speaking, we can regard $\Om (X)$ as a subspace  in $\Om_u(X^{reg})$. 
\end{lemma}

\subsection{Examples  of smooth structures  on stratified spaces}

\begin{example}\label{ex1}  Assume that  $X$ is a   realization of a  compact polytope  in $\R^n$, i.e., $X$ is a stratified  space  such that  each  stratum  $S$  of dimension $k$ of $X$  is an open  disk in some   affine subspace of dimension $k$  in $\R^n$.  Then  $X$ has a natural  smooth structure induced from the standard  smooth structure  on  $\R^n$
 i.e. $SC^\infty (X) : = SC^\infty (\R^n)|_X$.  Indeed,  by  construction $C^\infty(X)$ is a germ-defined $C^\infty$-ring, hence  the  condition 1
in Definition \ref{smooths} is trivially satisfied.    Now let us prove the validity of the second  condition on  the product  smooth structure inductively on  the dimension  of $X$. Note that  the validity of the second condition is trivially satisfied, if $\dim X = 0$.  Since $X$ is a realization of  a polytope,
 the tubular neighborhood  of  any  point $x \in S^k$  in  $\R^n$  has the form $B_x \times c L(1)$, where $L(1)$  is the intersection  of $X$  with  the sphere $S^{n-k-1}$  of a small  enough radius   centred at $x$ on a hyperplane   through $x$ in $\R^n$ that is  orthogonally complement to  $S$, and $B_x$ is  an open ball around $x$ in $S$.
 By the dimension  induction assumption, $L(1)\subset  S^{n-k-1}$   has a natural smooth structure induced  from  the embedding $L(1) \to S^{n-k-1} \subset  \R^n$, which  satisfies  the conditions  of
 Definition \ref{smooths}, since  the  projection from a punctured sphere $S^{n-k-1}\setminus  \{ pt\}$  to $\R^{n-k-1}$, where $\{pt\} \not \in L(1)$,  sends  $L(1)$  to a   realization of  a polytope of lower dimension
 in $\R^{n-k-1}$, and this  projection  is a diffeomorphism  between $S^{n-k-1} \setminus  \{ pt\}$ and $\R^{n -k -1}$.   To study   the smooth structure  on $cL(1)$  we need the following
 \begin{lemma}\label{lem:prod}  Assume that $A \subset \R^n$ and $B\subset \R^m$. Then 
 the  sheaf $SC^\infty (\R^n \times \R^m)|_{A \times B}$  is  generated  by $\pi^*_1 (SC^\infty(\R^n) | _A)$ and $\pi^*_2 (SC^\infty(\R^m)|_B)$, where
 $\pi_1$ and $\pi_2$  are the projection of $\R^{n +m}$  onto $\R^n$ and $\R^m$ respectively.
 \end{lemma}
 \begin{proof} Lemma  \ref{lem:prod}  is  a  consequence  of the  simple  fact that  $SC^\infty (\R^n \times \R^m)$ is generated by  $SC^\infty(\R^n)$ and $SC^\infty(\R^m)$.
 \end{proof}
 
  Lemma \ref{lem:prod}  implies  that  the  induced  smooth  structure  on   the  product $B_x \times cL(1)$
  satisfies   the condition 2  of Definition  \ref{smooths}.   The  condition (3a)  of Definition \ref{smooths} trivially holds, and the condition  (3b) of Definition \ref{smooths}   also holds for the  smooth structure on the  cone $cL(1)\subset \R^{n-k}\subset \R^{n}$, using   partition of unity on $\R^{n-k}$.
   This   proves  that  the  induced  smooth structure  on  $X$ satisfies   all the conditions  of Definition \ref{smooths}.
\end{example}

\begin{example}\label{ex2} Assume that  $(X_i, C^\infty (X_i))$   are stratified  spaces  provided  with   smooth structures. Then it is easy to  verify that
$(\Pi X_i, \Pi (C^\infty (X_i))$  is a stratified
manifold  provided  with  a  smooth structure.
\end{example}

\begin{example}\label{ex3} Assume that we have a   continuous surjective map   $M \stackrel{\pi}{\to} X$
from  a  smooth manifold $M$  with corner   to a stratified space $X$  of depth 1 such that
for each stratum $S_i \subset  X$ the triple  $(\pi ^{-1} (S_i), \pi_i, S_i)$  is a differentiable fibration,
moreover for each $x \in X^{reg}$ the preimage $\pi ^{-1}(x)$ consists of a  single point.  
Clearly  $\pi$ induces a   stratified  space structure  on $M$.
The $\R$-subalgebra  $C^\infty (X) : = \{ f \in  C^0 (X)|\, \pi ^* f \in C^\infty (M)\}$  will be called {\it a resolvable smooth structure}.  We  are going to show that a resolvable  smooth structure satisfies  the conditions in  Definition \ref{smooths}.
First, $C^\infty (X)$ is  a germ-defined $C^\infty$-ring, since $C^\infty (M)$ possesses this property.  
Next,  the existence of a local smooth trivialization $\phi_x$  for each $x \in X$, which satisfies   the  last two conditions  of Definition \ref{smooths}  is a consequence of  the  existence of a differentiable
fibration $(\pi ^{-1} (S_i), \pi_i, S_i)$ and the fact  that $\pi$ induces  a  stratified space structure on $M$.
The space $M$ will be called {\it a resolution} of $X$.
\end{example}

In what follows we study some properties of  a resolvable smooth structure on a  stratified space of depth 1.

We say that  $C^\infty(M)$ is {\it locally smoothly contractible},  if for  any  $x\in M$ there exists  an open neighborhood $U(x) \ni x$ together
with a smooth homotopy $\sigma : U (x) \times  [0,1] \to  U(x)$ joining the   identity map with the  constant map $ U(x) \mapsto x$ \cite[\S 5]{Mostow1979}.

A $C^\infty$-ring $C^\infty (X)$ is called  {\it finitely generated}, if  there are finite elements $f_1, \cdots, f_k \in C^\infty (X)$ such that any $h \in C^\infty (X)$ can be written as $ h = G (f_1, \cdots, f_k)$, where
$G \in C^\infty (\R^k)$. 

\begin{lemma}\label{canres}  Every pseudomanifold $X$ with edges has a resolvable smooth structure, which is locally smoothly contractible.   
\end{lemma}

\begin{proof}  By definition (see Example \ref{edge}.1), there exist  a compact smooth manifold $M$ with boundary  $\p M$ and   a surjective  map $\bar \pi: M \to X$.  In Example \ref{ex1}.3  we have shown   that  such a $X$ has a resolvable smooth structure  $C^\infty (X) : = \{ f \in  C^0 (X)|\,\bar \pi ^* f \in C^\infty (M)\}$.  We will show that $C^\infty (X)$ is locally smoothly contractible.
Let $S_i$ be a singular stratum of $X$, and $\bar \pi ^{-1} (S_i)  = \p M _i \subset \p M$.
Let $V(\p M_i)$ be  a collar open neighborhood of $\p M_i$ in $M$. Then $U(S_i) : = \bar \pi ( V(\p M_i))$ is an open neighborhood of $S_i$ in $X$. Let us consider the following  commutative diagram

$$\xymatrix{I \times  V(\p M_i) \ar[d]^{(Id\times \bar\pi)}\ar[r]^{\tilde F} & V(\p M_i)\ar[d]^{\bar\pi}\\
I\times U(S_i) \ar[r]^{F}  & U(S_i)
}$$
where  $\tilde F$ is a  smooth  retraction
from $V(\p M_i)$ to $\p M_i$, constructed using the  fibration $[0,1) \to V(\p M_i) \to \p M_i$.  
We set
$$ F ( t, x) : = \bar \pi (\tilde  F(t, \bar \pi^{-1}(x) )).$$
Since $\tilde  F_{| \p M_i} = Id$, the map $F$ is well-defined.   Clearly $F$ is a smooth homotopy, since $\tilde F$ is a smooth homotopy.  This proves  Proposition \ref{canres}.
\end{proof}

\begin{proposition}\label{infg} A resolvable
smooth structure on a stratified space $X$  of depth 1 obtained from a  smooth manifold $M$ with corner  is  not finitely generated
as a $C^\infty$-ring, if there exists $x\in X$ such that $\dim \pi^{-1}(x)\ge 1$, where   $\pi :  M\to X$  is the  associated projection.
\end{proposition}

\begin{proof}
Assume the opposite i.e.  $C^\infty (X)$ is  generated by $g_1, \cdots, g_n\in C^\infty (X)$. Then  $G:= (g_1, \cdots , g_n)$ defines a 
smooth embedding $X \to \R^n$. Hence   $C^\infty(X) = C^\infty (\R^n) /I$, where $I$ is an ideal of $C^\infty (\R^n)$   of smooth functions on $\R^n$ vanishing on $G(X)$ \cite[p. 21, Proposition 1.5]{MR1991}.  In particular, 
the  cotangent space $T^*_x X$ is a finite  dimensional linear space for all $x \in X$.  We will show that this assertion leads to a contradiction.
 
Let $S$ be a stratum  of $X$  such that $\dim (\pi^{-1} ( S)) \ge \dim S + 1$.  
Let $ x\in S$  and $U(x)$  a  small open neighborhood of $x$ in $X$. 
Let $f \in C^\infty (U(x))$, equivalently $\pi^*(f) \in C^\infty(\pi ^{-1} (U(x))$. Let $\chi : \pi ^{-1}( U(x) )\to \R^p_+ \times
\R^{n-p}\subset \R^n$ be a coordinate map on  $\pi ^{-1} (U(x))\subset M$. By definition of manifolds with corner, we have $(\chi^{-1})^* \pi^*(f) \in C^\infty (\tilde U)$ for some open set $\tilde U \subset \R^n$ containing $\chi (\pi ^{-1} (U(x)))$.    Denote by $\tilde S$ the  preimage  $\chi \circ \pi^{-1}( S\cap U(x))$, which is a submanifold of $\tilde U(x)$. Let us denote the restriction of $ \pi \circ \chi ^{-1}$ to $\tilde S$   by $\tilde\pi $. Then   the triple $(\tilde S, \tilde \pi, S\cap U)$ is a smooth fibration, whose fiber $\tilde \pi^{-1} (y)$ is   a smooth manifold of dimension at least 1.  
We note that  $(\chi ^{-1})^*\pi ^*(f)$ belongs to the subalgebra $C^\infty ( \tilde U,\tilde S,\tilde \pi)$ consisting of
smooth functions on $\tilde U$ that are constant along fiber  $\tilde\pi ^{-1}(y)$  for all $y\in S\cap U(x)$.
Since   the depth  of $X$  is 1, $C^\infty (\tilde U,\tilde S,\tilde \pi)$ is identified with the set of  smooth  functions on $U$. 

Shrinking $U(x)$    we can assume that  $\tilde S = \tilde U \cap \R^k$  and $\tilde \pi:\tilde S \to S \cap U$ is the restriction of a linear projection  $\bar \pi:\R^k  \to \R^l$, where $ l = \dim S$, $k = \dim \tilde S$, and $S \cap U = \R ^l \cap U$.  Here
we assume that $U$ is an open set in $\R^n$.  Let $\R^{n-k}$ with coordinate $\tilde x=(\tilde x^1, \cdots, \tilde x^{n-k})$ be  a complement to $\R^k$ in $\R^n \supset \tilde U$, and let  $\R^{k-l}\subset \R^k$  with coordinate $\tilde y=(\tilde y^1, \cdots, \tilde y^{k-l})$ be   the set $\bar \pi^{-1}(0)$. We  also equip the subspace $\R^l$   with coordinate $\tilde z = ( \tilde z^1, \cdots, \tilde z ^l)$.  
The condition $\dim \pi^{-1} (x) \ge 1$ in Proposition \ref{canres} is equivalent to the equality $k-l \ge 1$; in other words, $y$ is an essential variable.   Furthermore, a point $\tilde s \in \tilde S\subset \R^n$  has  (local) coordinates with $\tilde x =0$.

\begin{lemma}\label{local}  A function  $g\in  C^\infty (\tilde U)$  belongs to
$ C^\infty (\tilde U,\tilde S, \tilde \pi)$ if and only if $g$ has  the form
$$ g(\tilde x^1, \cdots, \tilde x^{n-k}, \tilde y,\tilde z) = \tilde x^1 g_1(\tilde x,\tilde y,\tilde z) + \cdots  + \tilde x^{n-k} g_{n-k}(\tilde x, \tilde y,\tilde z) + c(\tilde z), $$
where $g_i \in  C^\infty ( \tilde U)$  and $c(\tilde z)$ is a smooth function on $U$  depending only on variable  $\tilde z$.
\end{lemma}

\begin{proof}  We write for $g\in C^\infty (\tilde U,\tilde S, \tilde \pi)$
$$ g(\tilde x,\tilde y,\tilde z)  - g(0,\tilde y, \tilde z) = \int_0 ^ 1 {dg(t\tilde x, \tilde y,\tilde z)\over dt} dt= \int _0 ^ 1\sum_{i=1}^{n-k} { \p g(t\tilde x^1, \cdots, t\tilde x^{n-k},\tilde y,\tilde z)\over \p(t\tilde x^i)}\tilde x_i \, dt. $$
Setting
$$g_i := \int _0^1{ \p g(t\tilde x^1, \cdots, t\tilde x^{n-k}, \tilde y,\tilde z)\over \p (t\tilde x^i)}\,dt, $$
 we obtain  $g(\tilde x,\tilde y,\tilde z) = \sum_{i =1}^{n-k} \tilde x^i g_i (\tilde x,\tilde y,\tilde z) +  g(0,\tilde y,\tilde z)$.  Since $g(0,\tilde y, \tilde z)$ depends only on  $\tilde z$, we obtain  the ``only if" 
 part of Lemma \ref{local} immediately.  The ``if" part   is trivial. This proves Lemma \ref{local}.
\end{proof}

Now let us complete the proof of Proposition \ref{infg}.  Take a point $s\in S$  and a point $\tilde s \in \tilde \pi ^{-1} (s)$  such that $\tilde x (\tilde s) = \tilde y (\tilde s) = \tilde z(\tilde s) = 0$. Since $X$  has depth 1,  Lemma \ref{local}  implies that the maximal ideal $\sm _s$  is a linear space generated by  functions of the form $\tilde x^ig_{i, \alpha}(\tilde x,\tilde y,\tilde z), i = \overline{1,n-k}$.
Let us consider the sequence $S: = \{ \tilde x^1 \tilde y^1, \cdots, \tilde x^1(\tilde y^1)^m\in \sm _s\}$, $ m \to \infty$.  If  $\dim T^*_s X = \dim \sm_s /\sm_s^2 = n$,  there exists a   subsequence
$\tilde x^1(\tilde y^1)^{k_1}, \cdots, \tilde x^1 (\tilde y^1) ^{k_n}$  such that  $\tilde x^1(\tilde y^1) ^ m$ is a linear combination of $\tilde x^1(\tilde y^1)^{k_j}$  modulo $\sm_s ^2$ for any $m$, which is impossible. This completes the proof of Proposition \ref{infg}.
\end{proof}

\begin{remark} Proposition \ref{infg} partially answers the question 2 we posed in \cite[\S 5]{LSV2010}.
  We observe that there are
many  quotient smooth structures which are finitely generated, i.e.  a quotient   by a  smooth group action.  In this case
the dimension of the fiber over singular strata (e.g. the dimension of a singular orbit) is smaller than or equal to the dimension of the  generic fiber (the dimension of a generic orbit, respectively).
\end{remark}


\section{Symplectic  stratified spaces and  compatible  smooth structures}\label{s4}

In this chapter we introduce the notion of a  stratified symplectic space $(X,\om)$  (Definition \ref{stsym}), which is close to that  introduced by Sjamaar and Lerman \cite{SL1991} (Remark \ref{compsl}). We also introduce the notion of a  weakly symplectic smooth  structure 
 and the notion  of a Poisson smooth structure on $(X,\om)$  (Definition \ref{spos}).  We give examples  of weakly  symplectic smooth structures  and Poisson smooth structures (Propositions \ref{quote}, \ref{pois1}, Examples \ref{nil}.1-3).  We   prove the existence
and uniqueness of a Hamiltonian flow    associated with a smooth function $H$  on a  symplectic stratified space  $X$, which is
equipped with a  Poisson smooth structure  (Theorem \ref{ham}).  We  compare  our result with a
result by Sjamaar and Lerman  in \cite[\S 3]{SL1991} (Remark \ref{hsl}). We prove that the Brylinski-Poisson homology of a symplectic stratified space  $X$ provided with a  Poisson  weakly symplectic smooth structure  is isomorphic to the  de Rham cohomology of  $X$, if the regular strata   of $X$ have the same dimension  (Theorem \ref{hodge}.)  
 Then we show  that, under  a mild condition,  a   stratified  symplectic  space $(X, \om)$  provided  with a Poisson  weakly symplectic  smooth structure $C^\infty (X)$  enjoys   many  nice properties   related   to the existence of a Lefschetz decomposition (Lemma \ref{stable}, Proposition \ref{hlv}, Theorem \ref{harm}).

\subsection{Symplectic stratified  spaces}
\begin{definition} \label{stsym}  A stratified space $X$ is called {\it symplectic}, if every stratum  $S_i$ is provided with a symplectic form $\om_i$. The collection $\om:=\{\om_i\}$ is called {\it a stratified symplectic form}, or  simply {\it a symplectic form}, if no  misunderstanding can occur.   
\end{definition}

\begin{remark} \label{compsl} Definition \ref{stsym}  coincides with  the first  topological  condition   in \cite[Definition 1.12.(i)]{SL1991} of a   symplectic   stratified space $X$ introduced by Sjamaar and Lerman. (The other  conditions \cite[Definitions 1.12.(ii), 1.12.(iii)]{SL1991}  require  the existence of a compatible smooth structure   on $X$, which is also called by other authors \cite{Hueb2004} a stratified  symplectic Poisson algebra). Thus any
 stratified symplectic space in Sjamaar's and Lerman's definition is  a stratified symplectic space  in our definition.  
\end{remark}

\subsection{Weakly symplectic smooth structures  and Poisson smooth structures}
On each  symplectic stratum $(S_i, \om_i)$ we define the bivector $G_{\om_i}$ to be the section of the bundle $\Lambda ^2 TS_i$ such that $ G_{\om_i}(x) = \p y_1 \wedge \p x_1 + \cdots  + \p y_n \wedge \p x_n$
if $\om_i(x) = \sum_{j=1}^n dx^j \wedge dy^j$ \cite[\S 1.1]{Brylinski1988}. If we regard $\om_i$ as  an element in $End (TS_i, T^*S_i)$ and $G_{\om_i}$ as an element in $End (T^*S_i, TS_i)$, then $G_{\om_i}$ is the inverse of $\om_i$. The bi-vector $G_{\om_i}$ defines  a Poisson bracket on $C^\infty (S_i)$ by setting $\{f, g\}_{\om_i} := G_{\om_i} (df\wedge dg)$.

\begin{definition}\label{spos}  Let $(X, \om)$ be a symplectic stratified space and $C^\infty (X)$ be a smooth structure on $X$.  

1. A smooth  structure $C^\infty (X)$  is  said  to be  {\it weakly symplectic}, if there is a smooth 2-form
$\tilde \om \in \Om ^2 (X)$ such that   the restriction of $\tilde \om$ to each stratum $S_i$ coincides with $\om_i$.
In this case we also say that $\tilde \om$ is {\it  compatible  with $C^\infty (X)$}. 

2.  A smooth structure $C^\infty (X)$ is called {\it Poisson}, if there is a Poisson structure $\{ , \}_\om$ on  $C^\infty (X)  $   such that  $(\{f, g\} _{\om})_{| S_i} =
\{f_{|S_i}, g_{|S_i}\} _{\om_i}$  for any stratum $S_i\subset X$.
\end{definition}

\begin{remark}\label{pois0}  1. Lemma \ref{injf}  implies  that there exists  at most one 2-form $\tilde \om \in \Om^2 (X)$  which is
compatible with a given smooth structure $C^\infty (X)$.

2. We claim that the condition 2 in Definition  \ref{spos} is equivalent  to the existence of  a smooth Zariski bi-vector field $\tilde G_\om \in \Gamma(\Lambda ^2 T^Z (X))$  such that 
\begin{equation}
\tilde G_\om(\alpha) _{| S_i}  = G_{\om_i} ( \alpha_{ |S_i} )\label{biv0}
\end{equation}
for any stratum $S_i \subset X$.
Indeed, the
existence  a section $\tilde G_\om$  satisfying (\ref{biv0})  defines  a Poisson structure on $C^\infty (X)$ by setting
$\{f, g\}  (x) : = \tilde G_\om (df \wedge dg)(x)$. Conversely, assume that there is a Poisson structure $\{, \}_\om$
on $C^\infty (X)$ whose restriction to  each stratum $S_i$ coincides with  the given Poisson  structure on $S_i$. We claim
the bi-vector $\tilde  G_{\om_i}$ is a smooth Zariski bi-vector field. Since the  space of
smooth differential forms is germ-defined,  it suffices to show  the above claim locally. Note that on some neighborhood $U$ we  can write $\Om^2 (X) \ni\alpha = \sum _if_i dg_i \wedge dh_i$, where $f_i, g_i, h_i \in C^\infty (U)$.  Since the smooth structure is Poisson, we get  
\begin{equation}
\tilde G_\om (\alpha) =  \sum _i\tilde G_\om (f_i dg_i \wedge dh_i) = \sum_i f_i \{ g_i, h_i\} \in C^\infty (U). \label{biv1}
\end{equation}
This proves our claim. 

3. The condition 2  of Definition \ref{spos} agrees with  the condition (iii)  in Definition 1.12 of \cite{SL1991}  by Sjamaar and Lerman  of a stratified symplectic Poisson
algebra. It also agrees  with our Definition of a Poisson smooth structure  on  a  conical symplectic pseudomanifold  in \cite[\S 4]{LSV2010}.

\end{remark}

Now we are  going to  consider  important examples  of weakly symplectic  smooth structures and  Poisson  smooth structures.
\begin{example}\label{quot}  We assume that    a compact Lie group $G$ 
acts  on a  connected symplectic manifold $(M,\om)$  with  proper moment map $J : M \to \g ^*$.   Let $Z = J^{-1} (0)$.
The quotient space $M_0 = Z /G$ is called a  symplectic reduction of $M$.  If  $0$ is  a singular value of  $J$ then $Z$ is not a manifold and $M_0$ is  called {\it a singular symplectic reduction}.  It is known that $M_0$ is a stratified
symplectic space in Sjamaar's and Lerman's definition  \cite{SL1991},  and hence in our definition,  see Remark \ref{compsl}. Let us recall   the  description of $M_0$ by Sjamaar and Lerman.  For a subgroup  $H$ of $G$ denote by $M_{(H)}$ the set of
all points whose stabilizer is conjugate  to H, the stratum  of $M$ of orbit type $(H)$.

\begin{lemma}\label{sl1}\cite[Theorem 2.1]{SL1991} Let $(M, \om)$ be a Hamiltonian  $G$-space with moment map
$J : M \to \g ^*$.  The intersection  of the stratum $M_{(H)}$ of orbit type  $(H)$ with the zero level set $Z$  of the moment map
is a manifold, and the orbit space
$$ (M_0)_{(H)} =  (M_{(H)} \cap Z) /G$$
has a natural  symplectic  structure $(\om_0)_{(H)}$ whose  pullback to $Z_{(H)} : =
M_{(H)} \cap Z$ coincides with the restriction to $Z_{(H)}$ of the  symplectic form $\om$  on $M$.  Consequently the stratification
of $M$ by orbit types induces a decomposition of the reduced space $M_0 = Z/G$ into a disjoint union  of symplectic manifolds
$M_0 = \cup_{H \subset G } (M_0)_{(H)}$.
\end{lemma}

Since $J$ is proper, by Theorem 5.9 in \cite{SL1991} the  regular part $M_0^{reg}$ is connected. Sjamaar and Lerman also  defined 
a ``canonical" smooth structure  on  $ M_0$    as follows.  
Set $ C^\infty (M_0)_{can}: =  C^\infty  (M)^G /I ^ G$, where $I^G$  is the ideal  of $G$-invariant  functions vanishing on $Z$ \cite[Example 1.11]{SL1991}.  We will show that $C^\infty (M_0) _{can}$   is also a smooth structure in the sense of
 Definition \ref{smooths}.
Denote by $\pi$ the natural projection $ Z \to  Z/G$.  Set  $C^\infty (Z):= C^\infty (M) |_Z$. 
 Since $Z$ is closed, $C^\infty (Z) = C^\infty (M) / I_Z$, where $I_Z$ is the  ideal  of smooth functions on $M$ vanishing
on $Z$.

We claim that the space  $C^\infty  (Z) ^ G$ of $G$-invariant smooth functions on $Z$ can be identified with the   space  $ C^\infty (M_0)_{can} =  C^\infty  (M)^G /I ^ G$.   Clearly  $C^\infty (M) ^G/I^G$  is a  subspace of  $G$-invariant smooth functions on $Z$. On the other hand, any smooth function $f$ on $M$ can be modified to a $G$-invariant smooth  function $f_G\in C^\infty (M)$ by setting
$$f_G (x): = \int _G f(g\cdot x) \mu _g$$
for  a $G$-invariant   measure $\mu _g$ on $G$    normalized by  the condition $vol(G) = 1$.
So if $ g  \in C^\infty(Z)^G$, then $g$ is the restriction of  a $G$-invariant function   on $M$.
In other words, we have an injective map  $C^\infty  (Z) ^ G \to  C^\infty  (M)^G /I ^ G$.  Hence  follows the identity
$C^\infty  (Z) ^ G =   C^\infty (M_0)_{can}$.  It  follows that $C^\infty (M_0)_{can}$ is the quotient   of the smooth
structure obtained  from $C^\infty (Z)$ via the projection $\pi: Z \to  M_0$.  In particular, $C^\infty (M_0)_{can}$ is
a  germ-defined  $C^\infty$-ring, since $C^\infty (Z)$ is a germ-defined $C^\infty$-ring. 

\begin{proposition}\label{lem:smoothsl} $C^\infty (M_0) _{can}$ is a smooth structure  in the sense of Definition \ref{smooths}.
\end{proposition}
\begin{proof} Note that   $C^\infty (M_0) _{can}$ satisfies the first  condition in Definition
\ref{smooths}.    
To  show that  $C^\infty (M_0) _{can}$ satisfies the other conditions in Definition
\ref{smooths}, we use  Theorem 6.7  in \cite{SL1991} which  asserts that  there  is a proper  smooth  embedding
$(M_0,C^\infty (M_0) _{can}) \stackrel{i}{ \to} (\R^n, C^\infty (\R^n))$  such that  the image  $i (M_0)$ is  a stratified  Whitney  subspace $X$ of $\R^n$.
In particular $C^\infty (M_0)_{can} = i^* (C^\infty (\R^n)).$  

 We describe   a neighborhood   of a point $p\in X$ in $\R^n$  following \cite[p. 410]{SL1991}.
Let $S$ denote the stratum of $X$ that contains $p$.  Let $N'$ be a submanifold  in $\R^n$ that is transversal to each stratum  of $X$, intersects  $S$ in the single point $p$
and satisfies $\dim N' + \dim S = n$.  Let $B_\delta(p)\subset \R^n$ denote the ball of radius $\delta$. By Whitney's condition B, if $\delta$ is sufficiently small, then the sphere $\p B_\delta(p)$  will be transversal to  to each stratum in $X\cap N'$. Fix such a $\delta >0$. Next we consider the {\it normal slice} $N(p):=  N' \cap X \cap B_\delta(p)$  and {\it the link} $L(p): =N' \cap X \cap \p B_\delta(p)$ of the stratum $S$  at the point $p$.  These spaces are canonical Whitney stratified spaces, since they are transversal intersections
of Whitney stratified spaces. 
 Furthermore, $S$ has an open  neighborhood  $T_S$ in $X$ with local trivial  fibration
$\pi: T_S \to S$ such that  $\pi^{-1} (p) $ is homeomorphic to   $cL(p)$. Using  Lemma \ref{lem:prod}, this  implies that the induced  smooth structure
on $X$  satisfies the condition 2 of Definition \ref{smooths}.   The last two conditions (3a)  and (3b)  of Definition \ref{smooths}  also hold   for $X$, since  locally we have the same  description of $X$ as  that in Example \ref{ex1}.
This completes the  proof of  Proposition \ref{lem:smoothsl}. (We also refer  the reader to a more explicit, algebraic  description of the local structure  of $C^\infty(M_0)_{can}$  given in  Theorem 5.1 \cite{SL1991}  and its  proof).
\end{proof}

\begin{proposition}\label{quote}The  smooth structure  $ C^\infty (M_0)_{can}$   is both  weakly symplectic  and Poisson.
\end{proposition}
\begin{proof} 
We observe that $C^\infty (M_0)_{can}$ is   weakly symplectic,  since  by Proposition \ref{sl1} the pull back $\pi ^{*} (\om_0)$ is equal to  the restriction of the symplectic form $\om$ to $Z$. Furthermore, the Poisson property  of $C^\infty (M_0)_{can}$    follows  from \cite[Proposition 3.1]{SL1991}, where they showed that  $C^\infty (M_0)_{can}$  is closed under  the Poisson bracket.
This completes  the proof  of Proposition \ref{quote}.
\end{proof}
\end{example}

Let us consider  another  important  class  of  stratified   symplectic spaces, which are the closure of   nilpotent orbits in a complex semisimple Lie algebra
$\g$. This class has been examined by Panyshev \cite{Panyushev1991}, Huebschmann
\cite{Hueb2004}, Fu \cite{Fu2003}  and many other  under  different   perspectives.

\begin{example}\label{nil}  For $x\in \g$  let $x = x_s + x_n$  be the Jordan decomposition of $x$, where $x_n \not = 0$ is a nilpotent element, 
$x_s $ is a semisimple and $[x_s,x_n]=0$. Denote by $G$ the adjoint group of
$\g$  and by $\Zz_G(x_s)$ the centralizer of $x_s$ in $G$.   The adjoint orbit $G(x)$  is  a fibration over $G(x_s)$ whose fiber over $x_s$ is the $\Zz_{G} (x_s)$-orbit of $x_n$.  Since $G(x_s)$ is a closed orbit,  
a neighborhood $U$ of a point  $x\in \overline{G(x)}$ is isomorphic to   the product  $B \times \overline{\Zz_{G}} (x_s)\cdot x_n$, where
$B$ is   an open neighborhood  of $x_s$ in $G(x_s)$.  It is known that the closure $\overline{\Zz_{G} (x_s) \cdot x_n}$ is a finite union of  ${\Zz_{G}}(x_s)$-orbits  of nilpotent elements in the Lie subalgebra $\Zz_\g (x_s)$ \cite[chapter 6]{C-M1993}, so the closure $\overline{G(x)}$ is a  finite union of  adjoint orbits in $\g$ provided with the Kostant-Kirillov  symplectic structure.  Thus   $\overline{G(x)}$ is a decomposed space, whose strata are symplectic manifolds.  Moreover $\overline{G(x)}^{reg} = G(x)$ is  connected. 

1. Now assume that $x_n$ is a minimal nilpotent element
in $\Zz_{\g} (x_s)$. Then
$\overline{G(x)}$ is a stratified symplectic space of depth 1, since  $\overline{\Zz_{G} (x_s) \cdot x_n} = \Zz_{G}(x_s) \cdot x_n \cup \{0\}$, \cite[\S 4.3]{C-M1993}. The embedding
$\overline{G(x)} \to \g$ provides $\overline{G(x)}$ with a natural  finitely generated  $C^\infty$-ring  
$$C^\infty _1 (\overline{G(x)}):= \{f \in C^0(\overline{G(x)})|\, f = \tilde f_{|\overline{G(x)}}\text { for some } \tilde f \in C^\infty (\g)\}.$$ 
Clearly, $C^\infty _1 (\overline{G(x)})$  satisfies  the  first condition of Definition \ref{smooths}. The last two conditions in Definition \ref{smooths} also hold, since $\overline {G(x)}$ is a fibration over $G(x_s)$  whose fiber is  the cone
$\overline{\Zz_{G(x_s) }(x_n)}$  containing  the origin $\{0\} \in \g$.  Thus $C^\infty _1 ( \overline {G(x)})$  is a smooth structure  according to Definition \ref{smooths}. 

The smooth structure $C^\infty_1 (\overline{G(x)})$ is  Poisson that is inherited from the Poisson structure on  $\g$.
It is  also  weakly symplectic, since  the symplectic  form on $\overline{G(x)}$ is the restriction of the  smooth 2-form $\om _x (v, w) = \la x, [v, w]\ra$ on $\g$.  

2. We still assume that $x_n$ is minimal in $\Zz_{\g} (x_s)$.  In \cite[Lemma 2]{Panyushev1991} Panyushev showed that  $\overline{G(x)}$ possesses an algebraic (Springer's) resolution of the
singularity at $\{0\} \in \overline {G(x)} \subset \g$. We will show that this  resolution brings a  resolvable smooth structure $C^\infty _2 (\overline{G(x)})$.  First we recall the construction  in \cite{Panyushev1991}. 
Let 
\begin{itemize}
\item $h$ be a characteristic of  $x_n$ (i.e. $h\in \Zz_{\g} (x_s)$ is a semisimple element and $(h, x_n, y_n)\subset \Zz_{\g} (x_s)$ is an $\ssl_2$-triple);
\item $ \Zz_{\g} (x_s) (i): = \{  s\in \Zz_{\g} (x_s) | [h, s] = i s\}$;
\item $\n _2 (x_s):= \oplus _{i \ge 2} \Zz_{\g} (x_s)(i)$;
\item $P(x_s)$ denote a parabolic subgroup with the Lie algebra $lP(x_s) := \oplus_{i \ge 0}  \Zz_{\g} (x_s) (i)$ 
\item  $N_-$ - the connected  Lie subgroup of $\Zz_G (x_s)$ with Lie algebra \\
$lN_-(x_s) := \oplus _{i < 0}  \Zz_{\g} (x_s) (i).$ 
\end{itemize}
It is known that $\overline {P(x_s) \cdot x_n } = \n _2(x_s)$ and $\Zz_{\Zz_{G(x_s)}} (x_n) \subset P(x_s)$. Hence  there exists a natural map  
$$\tau : \Zz_{G}(x_s) * _{P(x_s)} \n_2(x_s) \to \overline{\Zz_{G}(x_s)\cdot x_n},\, \hspace{1cm} g* n \mapsto  gn,$$
 which is a resolution of the singularity 
of the cone $\overline{\Zz_{G}(x_s)\cdot x_n}$.   The resolution of  the $\overline {G(x)}$   is obtained by  considering
the fibration  $F$ over the orbit $G(x_s)$ whose fiber  over $x_s$ is  $\Zz_{G}(x_s) * _{P(x_s)} \n_2(x_s)$.   The map $\tau$  extends to a  map $\tilde \tau : F \to  \overline{G(x)}$ as follows. Denote  by $(x_s, y)$ the point  in the  fiber  over  $x_s\in  G(x_s)$ in  $F$ that is defined by   $y \in \Zz_{G}(x_s) * _{P(x_s)} \n_2(x_s)$.  Then
we set $ \tilde  \tau (x_s , y  ) : = \tau _{x_s} ( y).$

The resolvable smooth structure $C^\infty _2 (\overline{G(x)})$ is defined by $\tilde \tau$  as in Example \ref{ex1}.3.
    By Lemma \ref{canres} $C^\infty_2(\overline{G(x)})$ is locally smoothly contractible.  
 
3.  In addition, now  we assume that $x_s = 0$,  so $\Zz_G(x_s) = G$,  $P(x_s) = P$  and $\n_2(x_s) = \n_2$. In this case it  has been shown in \cite{Panyushev1991} that  $\tau ^* (\om)$ is a  smooth  2-form on  $G*_P \n_2$. It follows that $C^\infty_2 (\overline{G(x)})$ 
is   weakly symplectic. Panyushev also showed that $\tau ^* (\om)$ is symplectic  
if and only if $x$ is even. (We refer the reader to \cite{C-M1993} and \cite{FU2006} for   a detailed description of nilpotent orbits.)  
\end{example}
\begin{lemma}\label{lem:pois1}  Assume that $X$ is a stratified symplectic space  with isolated conical singularities and $(\tilde X, \tilde \om, \pi : \tilde X \to X)$  a
smooth resolution of $X$  such that $\tilde \om$ is a symplectic form  on $\tilde X$  and $\pi ^* (\om_{| X^{reg}}) 
= \tilde \om _{|\pi ^{-1} ( X^{reg})}$. If  for  each singular point $x \in  X$ the preimage $\pi ^{-1} (x)$ is a coisotropic  submanifold  in $\tilde X$, then
the  obtained resolvable smooth structure $C^\infty (X)$ is Poisson.
\end{lemma}

\begin{proof}  We define   a Poisson bracket  on $C^\infty (X)$ by  setting $\{ g,  f \}_\om (x)  : = \{ \pi^{*} g, \pi^{*} f\}_{\tilde \om}(\tilde x)$, for $\tilde x \in \pi ^{-1} (x)$.
We will show that this definition does not depend on the choice of a particular $\tilde x$.   By definition $\{ \pi^{*} g, \pi^{*} f\} _{ \tilde \om} (\tilde x): = G_{\tilde \om }(d{\pi^{*} g} , d{ \pi^{*} f})({\tilde x})$. Since $\pi^{*} f$ and $\pi^{*} g$ are constant along   the   coisotropic  submanifold  $\pi^{-1}(x)$, we get $G_{\tilde \om} (d{\pi^{*} g} , d{ \pi^{*} f})({\tilde x})=0$. This proves Lemma \ref{lem:pois1}.
\end{proof}

It has been showed in  \cite[\S 2]{Beauville2000}, \cite{Panyushev1991}  that the preimage $\tau ^{-1}(\overline{\Zz_{G}(x_s)\cdot x_n} \setminus \Zz_{G}(x_s)\cdot x_n)$  is  a  Lagrangian submanifold in  $\Zz_{G}(x_s) * _{P(x_s)} \n_2(x_s)$  which is the cotangent  bundle  $T^* (\Zz_{G}(x_s)/P(x_s))$  supplied with the natural symplectic  structure. Using Lemma \ref{lem:pois1} we  summarize   our examination of Example \ref{nil} in the following

\begin{proposition}\label{pois1} Let $x = x_n + x_s$  where $x_n$ is a minimal  nilpotent  element in $\Zz_\g (x_s)$.  Then $C^\infty_1 (\overline{G(x)})$ is a weakly symplectic  and Poisson smooth structure.   If $x_s = 0$, then $C^\infty_2(\overline{G(x)})$  is  a weakly symplectic and  Poisson smooth structure.
\end{proposition}

\subsection{The existence of Hamiltonian flows}
Let $(X, \om)$ be a stratified symplectic space  and $C^\infty (X)$  a  Poisson smooth structure on $X$.  For any $H \in C^\infty (X)$  
 we define  a linear operator   $X_H: C^\infty (X) \to  X^\infty (X)$  by
 $$X_H (f): = \{  f, H\} _\om \text{  for } f \in C^\infty (X).$$
 By  definition, for given $H$,  the value $X_H(f)(x)$ depends only on the value $df(x)$.  Hence  $X_H$ is a  section  of  the Zariski  tangent bundle    of $X$.
 We call $X_H$ {\it  the Hamiltonian vector field  associated with $H$}.
 
\begin{lemma}\label{hav} The Hamiltonian vector field $X_H$
 is a smooth Zariski vector field   on  $X$.  If $x$ is a point in  a stratum $S$, then  $X_H (x) \in T_x  S$.
 \end{lemma}

\begin{proof}  By   definition of a Poisson structure, the function $X_H (f)$ is smooth for all $f \in C^\infty (X)$.
Hence $X_H$ is a smooth Zariski vector field.  This proves the  first assertion of Lemma \ref{hav}. To prove the second
assertion it suffices to show that, if the restriction of a function $f\in C^\infty (X)$ to a neighborhood $U_S(x) \subset S$ of  a point $x \in S$ is zero, then  $X_H (f)(x)  = 0$. The last identity holds, since $X_H (f) (x)$ is equal to the  Poisson bracket of  the restriction 
of $H$ and $f$ to $S$.  This completes the proof.
\end{proof}

\begin{theorem}\label{ham} (cf. \cite[\S 3]{SL1991}) Given a  Hamiltonian function  $H\in C^\infty (X)$ and a point $x \in X$ there
exists  a unique  smooth curve $\gamma: (-\eps, \eps) \to X$ such that for any $ f\in C^\infty (X)$ we have
\begin{equation}
{d\over dt}   f (\gamma (t)) = \{  f, H\}.\label{eq:flow}
\end{equation}
\end{theorem}

\begin{proof}  Let $\Phi_t(x)$  be   the flow that is generated by $X_H|_S$ on  each stratum $S\subset X$. Since $C^\infty (X)$ is Poisson,  the validity  of (\ref{eq:flow}) for $\gamma(t): = \Phi_t$ follows  from Lemma \ref{hav}. This proves the existence of  a flow satisfying (\ref{eq:flow}).

Now let us prove the uniqueness of the   flow satisfying (\ref{eq:flow}), 
using  Sjamaar's and Lerman's argument  in \cite[\S 3]{SL1991}.  
Let $x\in X$ and $\gamma (t), \,  t\in (-\eps_1, \eps_1)$  be an integral curve of the equation (\ref{eq:flow})  with $\gamma _0 (0) =
x$.   We will show that  $\Phi _{-t }(\gamma (t)) = x$ for all $0\le t \le \min (\eps, \eps _1)$.   By Corollary \ref{germ}    
smooth functions on $X$ separate  points. Therefore it suffices  to show that  for all $t
\le \min (\eps, \eps _1)$ and for
all $f \in C^\infty(X)$ we have
\begin{equation}
f(\Phi_{-t }(\gamma _t (t))) = f(x).\label{id1}
\end{equation}
As in \cite[\S 3]{SL1991}, using (\ref{eq:flow}), we have 
$${d\over dt}  f( \Phi _{-t }(\gamma  (t))) = \{ H, f\}_\om  ( \gamma (t)) + \{ f, H\} _\om  (\gamma (t)) = 0.$$
 This implies (\ref{id1})  and completes  the proof of  Theorem  \ref{ham}.     
\end{proof}

\begin{remark} \label{hsl} 1. In Example  \ref{quot}  we  proved that  the smooth structure on  a  singular  symplectic reduced space $(M^{2n},\om)//G$ defined by Sjamaar and Lerman  in \cite{SL1991}   is  a Poisson smooth structure in sense of our definition. Thus,  their result on the existence of a Hamiltonian flow  on $(M^{2n},\om)//G$ in \cite[\S 3]{SL1991} is  a consequence  of  our Theorem \ref{ham}.

2. In \cite{SL1991} Sjamaar and Lerman  used a slightly different  method for their
proof of the existence  of a Hamiltonian flow on  the singular  symplectic reduced space $(M^{2n},\om)//G$. They looked at the  corresponding Hamiltonian flow  on  $M$ and showed that this flow
descends to a Hamiltonian flow on  the reduced space. 
\end{remark}

\subsection{Brylinski-Poisson homology}
In this subsection we extend  the study  of the Brylinski-Poisson homology   of  symplectic pseudomanifolds with isolated conical singularities
in \cite[\S 4]{LSV2010} to the case of  stratified symplectic spaces  $X$ equipped with a  Poisson smooth structure.

 Assume  that  $C^\infty (X)$  is a   Poisson smooth structure.
We consider {\it the  canonical complex}
$$\to \Om ^{n+1}(X) \stackrel{\delta}{\to}  \Om^n (X) \to ...  ,$$
where $\delta$ is  a  linear operator  defined as follows. Let $\alpha \in \Om (X)$  and $\alpha = \sum_j f_0  ^j df_1 ^j \wedge  df_p ^j$ be a local representation of $\alpha$ as in  Definition \ref{smoothf}.  Then we set (see \cite{Kozsul1985},  \cite{Brylinski1988}):
 $$ \delta (f_0 df_1 \wedge \cdots \wedge df_n) := \sum_{i =1}^n (-1) ^{i+1} \{ f_0 , f_i\}_\om df_1 \wedge \cdots \wedge  \widehat{df_i} \wedge  \cdots \wedge df_n $$
  $$ + \sum_{1\le i < j \le n} (-1)^{i+j}f_0 d \{ f_i, f_j\}_\om \wedge df_1 \wedge \cdots \wedge \widehat {df_i} \wedge  \cdots \wedge \widehat {df_j} \wedge \cdots \wedge df_n.$$
  

\begin{lemma}\label{Bry3} (cf. \cite[Lemma 4.3]{LSV2010})
1. $\delta  =  i(G_\om)\circ d - d \circ  i(G_\om)$. In particular,  $\delta$ is well-defined.\\
2. $\delta^2 = 0$. 
\end{lemma}
\begin{proof} Recall that  $i$ denotes the canonical inclusion $X^{reg} \to X$. 

1. Let $\alpha \in \Om (X)$, then $i^* \alpha \in  \Om _u (X^{reg})$.
Using   
$$\delta \circ  i^* = i^* \circ  \delta, \,  i^* \circ d = d \circ i^* ,$$ 
and the validity of the first assertion  of Lemma \ref{Bry3}  for any smooth Poisson manifold $M$
\cite[Lemma 1.2.1]{Brylinski1988},  we have
\begin{equation}
i ^ * (\delta \alpha)   = \delta (i^* \alpha) = i ^ * ( i(G_\om)  \circ d \alpha - d \circ i (G_\om) \alpha) .\label{eq:com}
\end{equation}
 By Lemma \ref{injf}, $i^*$ is injective,  hence  the above   equality implies the first assertion of Lemma \ref{Bry3}.

2.  The second statement  of Lemma \ref{Bry3} is proved in the same  way, using the  injectivity of $i^*$. This completes the  proof of Lemma \ref{Bry3}.
\end{proof}

The following theorem \ref{hodge}
generalizes  \cite[Corollary 4.2]{LSV2010} . 

\begin{theorem}\label{hodge} Suppose $(X, \om)$ is a stratified symplectic space equipped with a Poisson smooth structure $C^\infty (X)$  which is also  weakly symplectic. If all regular  strata of $X$  has the same dimension  $2n$, the  Brylinski-Poisson  homology  of the complex $(\Om(X) , \delta)$  is isomorphic
to  the de Rham cohomology of $X$ with   reverse grading : $H_k (\Om( X), \delta) = H ^{2n -k} (\Om(X) , d)$. If, moreover,  the smooth structure  $C^\infty (X)$ 
is locally smoothly contractible, $H_k (\Om( X), \delta)$ is equal to the
 singular cohomology $H^{2n-k}(X, \R)$.
 \end{theorem}
\begin{proof} 
 Using the injectivity of $i^*$, we  derive   from (\ref{eq:com})
 the following  formulas  for all $k$: 
\begin{equation}
H_k (\Om(X), \delta) = H_k (i^*(\Om(X), \delta) \text{ and }  H^k(\Om(X), d) = H^k(i^*(\Om(X), d).\label{eq:iso}
\end{equation}
    
Since $C^\infty (X)$ is weakly symplectic,  there  exists  a symplectic   form $\tilde \om\in \Om^2 (X)$  such that $i^*(\tilde \om) |_{S_i} = \om_i$  for any
stratum $S_i$. 
Set $vol: = \tilde \om ^n/n!$. 
 Let  $ \tilde G^k_\om$  be the  pairing: $\Lambda^k (T^*X) \times \Lambda ^k(T^*X) \to C^\infty (X)$  associated with $\tilde G_\om$ whose existence  is  shown in Remark \ref{pois0}.2. We define  a symplectic star operator 
 $*_\om: \Om^k (X) \to \Om ^{2n-k}(X)$  as follows (cf.  \cite[\S 2.1]{Brylinski1988}).
\begin{equation}
 *_\om : \Om ^k (X) \to \Om^{2n-k} (X),\,  \beta \wedge *_\om \alpha : =  \tilde G^k_\om  (\beta, \alpha) \wedge \frac {\tilde \om ^n} { n !},
\nonumber 
\end{equation}
for all $\alpha, \beta  \in  \Om ^k (X)$.  In particular, on singular strata,  the  image  of $*_\om$ is zero.

For the sake of simplicity   we also denote by $\om$ the restriction of $\om $ to $X^{reg}$. The following Proposition is  proved by repeating  the proof of Proposition 4.2 in \cite{LSV2010}
 word-by-word, so we omit its proof.  

\begin{proposition}\label{sstar}    We have $*_\om(i^* (\Om ^k (X))) =i^*( \Om ^{2n -k} (X))$. 
\end{proposition}

The first assertion of Theorem \ref{hodge} follows immediaately  from (\ref{eq:iso})  and Proposition \ref{sstar}.  The second assertion of Theorem \ref{hodge}   follows from \cite{Mostow1979}. This completes the proof of Theorem \ref{hodge}.
\end{proof}

\subsection{A Leftschetz decomposition}
The notion of a Leftschetz decomposition  on a   symplectic manifold $(M^{2n}, \om)$  has been introduced by Yan in \cite{Yan1996}, where he gives an alternative proof the  Mathieu theorem on harmonic cohomology classes of $(M^{2n}, \om)$  using the
Leftschetz decomposition. 
Roughly  speaking, a Leftschetz decomposition on a symplectic  manifold $(M^{2n},\om)$ is an $\ssl_2$-module -structure
of $\Om (M^{2n})$.  The Lie algebra $\ssl_2$ acting on $\Om(M^{2n})$ is generated by  linear operators  $L,L^*, A$
  defined  as follows.
$L$ is the wedge multiplication   by $\om$, $L^*  :=  i(G_\om)$, and $A = [L^* , L]$.

Now assume that $(X, \om)$ is a stratified symplectic space  provided with a Poisson smooth structure $C^\infty (X)$, which is  also weakly symplectic. Then $\Om (X)$ is stable  under $L, L^*$. Hence we  obtain immediately

\begin{lemma}\label{stable}  The  space $\Om (X)$ is  an $\ssl_2$-module, where $\ssl_2$ is  the Lie algebra
generated by $(L, L^*, A=[L,L^*])$.
\end{lemma}

The next Lemma concerns  the Leftschetz decomposition of $\Om_u(X^{reg})$.  
Recall that  $\alpha \in \Om_u (X^{reg})$ is  called {\it   primitive}, if  $L^* \alpha  =0$.
We define the dimension  function $\mathrm{d }: X\to \Z$ by setting
$\mathrm{d}(x): = \dim S_x$, where $S_x$ is the connected  component  of the stratum containing $x$.

\begin{lemma}\label{rec} 1. For any $\gamma \in \Om ^k_u(X^{reg})$  and any $r\in \Z$ we have 
$$[L^r, L^*]\gamma= (r (k-{\mathrm{d}\over 2}) + r (r-1)) L^{r-1}.$$
2. There exists a   function $c: \Z \times \Z \to \R$ such that for any    primitive form $\gamma \in \Om  ( X^{reg})$  we have 
$\gamma  = c(\mathrm{d}, k)\cdot (L ^*) ^k \circ ( L^k \gamma)$.
\end{lemma}

\begin{proof} The first assertion of Lemma \ref{rec} for $r=1$ is well-known, see \cite[Corollary 1.6]{Yan1996}.
For $r\ge 2$ we use the following formula
$$[L^r, L^*] = L [L^{r-1}, L^*] + [ L, L^*] L^{r-1},$$
which leads to the first  assertion of Lemma \ref{rec} by induction.

2. The second assertion of Lemma \ref{rec} is proved by applying the first assertion recursively.
\end{proof}

For any $k \ge 0$  set 
$$\Pp_{k} (X) : = \{ \alpha \in \Om ^{k} (X) |\, \alpha \text { is  primitive }\}.$$

\begin{proposition}\label{hlv} Assume that $C^\infty(X)$ is both Poisson  and weakly symplectic. If all regular strata  of $X$  have the same dimension  $2n$, then we have the following  Leftschetz  decomposition   for $k \ge 0$
\begin{equation}
\Om ^{k}(X) = \Pp_{k} (X) \oplus L (\Pp_{k-2} (X)) \oplus \cdots \,  .\nonumber
\end{equation}
\end{proposition}

\begin{proof} Using the Leftschetz decomposition  on symplectic  manifold,  for $\alpha \in \Om^k (X)$
we decompose $i^*(\alpha) \in i^*(\Om^{k} (X))$ as 
\begin{equation}
i^*(\alpha)= \alpha_p ^ {k} + L ( \alpha _p ^{k-2}) +   \cdots + L^{[(n-k)/2]} \alpha _p ^{n-k- 2[(n-k)/2]},
\label{hl1}
\end{equation}
where $\alpha_p ^ j$  are  primitive  forms in $\Om_u ^j (X^{reg})$.
To prove Proposition \ref{hlv}, using the injectivity  of $i^*$,
it suffices to show that $\alpha_p ^j$  are elements in $i^* (\Om (X))$.
Now let us consider  the decomposition  of $i^*(\alpha) \in \Om ^{k}_u (X^{reg})$.  We will show 
that  all terms $\alpha_p ^j$ can be obtained from  a  linear combination of  $L^p(i^*(\alpha)),  (L^*)^p (i^*(\alpha)), \, p \ge 0,$ inductively  on  the degree $j$.

 First we assume that $k$  is even, i.e. $ k =2q$,
hence $\alpha_p ^0 \in C^\infty_u (X^{reg})$.  Applying to the both sides of (\ref{hl1}) the operator $L^{n-q}$ we get
\begin{equation}
 L^{n-q} (i ^* (\alpha) ) = \om ^n \cdot \alpha _p ^0 \in i ^*(\Om^{2n} (X)).\label{p0}
 \end{equation}
By  Proposition \ref{sstar}, (\ref{p0}) implies that $\alpha_p ^0 \in   i^*(C^\infty (X))$, what is required to prove in the first induction step.

Now let us assume that $k = 2q +1$.  As in (\ref{p0}), we have
\begin{equation}
L^{n-q-1} (i^*(\alpha)) = \om ^{n-1} \cdot \alpha _p^1 \in i ^* (\Om ^{2n-1} (X)).
\label{p1}
\end{equation}
 Taking into account $L^* \alpha ^1_p = 0$, using Lemma \ref{rec} and (\ref{p1}), the term $\alpha_p^1$  can be obtained  from
$L^{n-q-1} (i^*(\alpha))$ by applying the operator $c_{n, 1}\cdot (L^*) ^{n-1}$.  Hence $\alpha _p ^1 \in i^* (\Om (X^{2n}))$,  which completes the next induction step.

Repeating this procedure, we get all   terms  $\alpha _p ^j$, which  belong to 
$i^* (\Om (X))$.  
This completes the proof of Proposition \ref{hlv}.
\end{proof}

Since $[L,d] = 0$  holds on $\Om_u (X^{reg})$   and $i^*(\Om (X))$ is  stable under the action of $d$ and $L$,
the  equality $[L, d] = 0$ also holds on $\Om (X)$.  In particular, the   wedge product with $[\om ^k] $ maps
$H^{l-k} ( \Om (X), d)$ to $H^{l+k} (\Om (X), d)$  for any $l \in \Z$. A   stratified symplectic space $(X^{2n},\om)$  of dimension $2n$ equipped with a  Poisson  weakly symplectic smooth structure $C^\infty (X^{2n})$  
is  said  {\it to satisfy the hard
Lefschetz condition}, if the cup product
$$[\om  ^k] :  H^{n -k} (\Om (X^{2n}),d) \to  H^{ n +k} (\Om (X^{2n}),d)$$
is surjective  for any $k \le n = {1\over 2} \dim  X^{2n}$.
A    differential form  $\alpha \in \Om (X^{2n})$ is called {\it harmonic}, if $ d\alpha = 0 = \delta \alpha$.
Let us abbreviate $H^*(\Om (X^{2n}), d)$ by $H^*_{dR} ( X^{2n})$.

\begin{theorem}\label{harm}  Let $(X^{2n}, \om)$ be a  stratified  symplectic  space and $C^\infty (X^{2n})$
  Poisson smooth structure  which is also weakly synmplectic. Assume that  all regular strata of $X$ have the same dimension $2n$. Then the  following two assertions are equivalent:
 \begin{enumerate}
 \item
 Any  cohomology class in $H^*_{dR}(X^{2n})$  contains  a harmonic  cocycle.
\item
 $(X^{2n}, \om)$  satisfies   the hard Lefschetz  condition.
\end{enumerate}
\end{theorem}

\begin{proof}  The proof of  Theorem \ref{harm}  for  smooth  symplectic manifolds by Yan in \cite[Theorem 0.1]{Yan1996}
can be repeated word-by-word.  For the convenience of the reader we outline a proof here.
Denote by $H^k_{hr} (X^{2n})$ the space of all harmonic $k$-forms on $(X^{2n},\om)$, and  let $H^*_{hr} = \oplus _{ i = 0} ^{2n} H^i _{hr} (X^{2n})$.

Now let us prove that the assertion (1) of Theorem \ref{hodge} implies the assertion (2) of Theorem \ref{hodge}.
We consider the following diagram
$$
\xymatrix{ H^{n-k}_{hr} (X^{2n}) \ar[d] \ar[r]^{L^k} & H^{n+k}_{hr} (X^{2n}) \ar[d]\\
H^{n-k}_{dR} (X^{2n}) \ar [r]^{L^k} &  H^{n+k}_{dR} (X^{2n}). 
}$$
Let us  recall that $i : X^{reg} \to X^{2n}$ is  the canonical inclusion.  Since $[L, \delta] = -d$ \cite[Lemma 1.2]{Yan1996}, which can be   easily proved  for  $(X^{2n}, \om)$ satisfying the condition  of Theorem \ref{harm}, 
Proposition \ref{hlv} implies that $L^k :  H ^{n-k}_{hr}(X^{2n})  \to  H^{n+k}_{hr} (X^{2n})$
is an isomorphism.  Since   the vertical arrows in the diagram are surjective,
we conclude that  the second horizontal arrow in the diagram is also surjective.
This proves (1)  $\LRA$ (2).

Now let us prove  that (2) $\LRA$ (1).  Note that the condition (2) implies that  \cite[\S 3]{Yan1996}
$$ H^{n-k}_{dR} (X^{2n}) = Im\, L + P_{n-k},$$
where $P_{n-k} := \{ \alpha \in H^{n-k}_{dR} (X^{2n}) |\, L^{k+1} \alpha = 0\in H^{n+k+2}_{dR} (X^{2n})\}.$

Using induction argument, it suffices to prove that in each primitive  cohomology class $v \in P_{n-k}$
there is a harmonic cocycle.  Let $v = [z]$, $z\in \Om ^{n-k} (X^{2n})$. Since $v$ is primitive we have $[z \wedge \om ^{k+1}] = 0\in H^{m +k +2}_{dR} (X^{2n})$. Hence, $z\wedge \om ^{k+1} = d\gamma$ for some $\gamma \in \Om ^{n +k +1} (X^{2n})$.  By Proposition \ref{hlv} the operator $L ^{k+1} : \Om ^{n-k-1} (X^{2n}) \to \Om ^{n+k +1} (X^{2n})$ is onto,
consequently there exists  $\theta \in \Om ^{m-k-1} (X^{2n})$ such that  $\gamma = \theta \wedge \om ^{k+1}$. It follows that
$(z-d\theta) \wedge \om ^{k+1} = 0$.  Therefore $w = z-d\theta$ is primitive and closed. Since $[L^*, d] = \delta$
\cite[Corollary 1.3]{Yan1996}, we obtain $\delta w =0$. This completes the proof of Theorem \ref{hodge}.

\end{proof}

\begin{remark}\label{cav} 1. If
$C^\infty (X^{2n})$ is  locally smoothly contractible, by \cite[Theorem 5.2]{Mostow1979}
the de Rham cohomology $H^* (\Om (X^{2n}), d)$ coincides with the singular cohomology  $H^* (M, \R)$), since
$ X^{2n}$   admits  smooth partitions of unity, see Proposition \ref{fin}. 

2. In \cite[Proposition 5.4]{Cavalcanti2005} Cavalcanti proved  that   the hard Lefschetz property
on a compact symplectic manifold implies   $Im \, \delta \cap \ker d = Im \, d \cap Im \, \delta$, see also \cite{Merkulov1998}. His theorem   can be
proved  word-by-word for stratified symplectic spaces satisfying the conditions  of Proposition \ref{hlv}, since the main ingredient  of the proof is    Proposition \ref{hlv}.

3. It is  interesting to know whether   we can extend  the   symplectic  cohomology theory  developed in  \cite{TY2009}, \cite{LV2011} to 
stratified symplectic  spaces satisfying the  conditions  of Proposition \ref{hlv}, since  the main  ingredient  of this theory is the existence of a  Leftschetz
decomposition.
\end{remark}

\section{Conclusions}
In this paper   we have worked  out a natural  refinement  of the concept of a smooth structure on a stratified   space, which is well suited
for the study of stratified symplectic  spaces.  In this refined concept  there  are two natural classes  of smooth structures on stratified
symplectic spaces:   weakly  symplectic   smooth structures  and Poisson smooth structures.  We show that  Poisson smooth structures on stratified symplectic spaces $X$, especially those  are also  weakly symplectic, enjoy  many  properties  of  symplectic manifolds, if the regular strata  of $X$ are of the same  dimension. We suggest to  study  stratified  symplectic  spaces,  whose  regular strata  are of varying  dimension,  for  future works.

\medskip

 {\bf Funding} This work  was supported   by RVO: 67985840  and by Grant of ASCR  Nr100190701 (H.V. L. and J. V.),  by MSM 0021620839 and GACR 20108/0397 (P.S.).
A part of this paper was written while H.V.L.  was visiting the ASSMS, GCU, Lahore-Pakistan. She thanks  ASSMS for their hospitality and  financial support.

{\bf Acknowledgement}  
 H.V.L.  thanks  Dmitry Panyushev  and Nguyen Tien Zung for  discussions, which have  helped to shape the final form of this paper.

\bigskip
\author{H\^ong V\^an L\^e},
\address{Institute of Mathematics of ASCR,
Zitna 25, 11567 Praha 1, Czech Republic}\\
\author{Petr Somberg},
\address{ Mathematical Institute,
Charles University,
Sokolovska 83,
180 00 Praha 8,
Czech Republic.}\\
\author{Ji\v ri Van\v zura},
\address{Institute of Mathematics of  ASCR, Zizkova 22, 61662 Brno, Czech Republic}

\end{document}